\documentclass[a4paper,12pt]{article}
\usepackage{amsmath,amsthm,amsfonts,amssymb,bbm}
\usepackage{graphicx,psfrag,subfigure}
\usepackage{cite}
\usepackage{hyperref}
\usepackage[hmargin=1in,vmargin=1in]{geometry}
\hypersetup{colorlinks=true}

\renewcommand{\d}{\mathrm d}

\newcommand{\R}{\mathbb R}

\newcommand{\wt}{\widetilde}

\renewcommand{\Re}{\operatorname{Re}}
\renewcommand{\Im}{\operatorname{Im}}
\newcommand{\Ai}{\operatorname{Ai}}
\newcommand{\Bi}{\operatorname{Bi}}
\newcommand{\Hi}{\operatorname{Hi}}

\newcommand{\A}{\mathcal A}

\renewcommand{\O}{\mathcal{O}}
\renewcommand{\P}{\mathbf P}
\newcommand{\E}{\mathbf E}
\newcommand{\I}{\mathrm i}

\newtheorem{proposition}{Proposition}[section]
\newtheorem{theorem}[proposition]{Theorem}
\newtheorem{lemma}[proposition]{Lemma}

\theoremstyle{definition}

\numberwithin{equation}{section}

\author{Patrik L.\ Ferrari\thanks{Institute for Applied Mathematics, Bonn University, Endenicher Allee 60, 53115 Bonn,
Germany. E-mail: {\tt ferrari@uni-bonn.de}} \and
B\'alint Vet\H o\thanks{Department of Stochastics, Budapest University of Technology and Economics;
MTA\,--\,BME Stochastics Research Group, Egry J.\ u.\ 1, 1111 Budapest, Hungary. E-mail: {\tt vetob@math.bme.hu}}}

\title{Upper tail decay of KPZ models\\ with Brownian initial conditions}

\date{July 24, 2020}

\begin{document}

\maketitle
\begin{abstract}
In this paper we consider the limiting distribution of KPZ growth models with random but not stationary initial conditions introduced in~\cite{CFS16}. The one-point distribution of the limit is given in terms of a variational problem. By directly studying it, we deduce the right tail asymptotic of the distribution function. This gives a rigorous proof and extends the results obtained by Meerson and Schmidt in~\cite{MS17}.
\end{abstract}

\sloppy

\section{Introduction}
The Kardar--Parisi--Zhang (KPZ) universality class of stochastic growth models in one dimension are described by a stochastically growing interface parameterized by a height function.
For general initial conditions, the one-point distribution of the large time limit can be written in terms of a variational problem.
The ingredients are the (scaled) initial condition and the Airy$_2$ process, $\A_2$.
The latter arises as the limiting interface process when the macroscopic geometry of the interface in the law of large numbers is curved.
It was discovered in the work of Pr\"ahofer and Spohn~\cite{PS02} and described by its finite-dimensional distribution.
Soon after, Johansson showed weak convergence of the discrete polynuclear growth model to the Airy$_2$ process~\cite{Jo03b}.
In the same paper, a first variational formula appeared (see Corollary~1.3 of~\cite{Jo03b})
\begin{equation}\label{eq1}
F_{\rm GOE}(2^{2/3}s)=\P\left(\sup_{t\in\R}(\A_2(t)-t^2)\leq s\right)
\end{equation}
where $\A_2$ is the Airy$_2$ process and $F_{\rm GOE}$ is the GOE Tracy--Widom distribution function discovered in random matrix theory~\cite{TW96}.
Formula~\eqref{eq1} corresponds to the flat initial condition as $F_{\rm GOE}$ is the limiting distribution of the corresponding rescaled interface.

Later, variational formulas describing the one-point distributions for some special initial conditions appeared in several papers, see for instance~\cite{QR13B,BL13,QR13,QR16}.
The first study of a large class of initial conditions, including random initial conditions, is the paper of Corwin, Liu and Wang~\cite{CLW16}.
In a last passage percolation model they showed the convergence of the one-point distribution to a probability distribution expressed by the variational formula
\begin{equation}\label{eq2}
\P\left(\sup_{t\in\R}\left\{h_0(t)+\A_2(t)-t^2\right\}\le s\right)
\end{equation}
where $h_0$ is the scaling limit of the initial height profile.
Shortly after, Remenik and Quastel in~\cite{QR16} asked and answered the question how much discrepancy from the perfectly flat initial condition would be allowed to still see the GOE Tracy--Widom distribution for the KPZ equation. In their paper the variational representation plays an important role.
The variational formula approach is proved to be useful since it allows one to go beyond the use of exact formulas and to show, for instance, universal limiting distribution for a flat but tilted profile~\cite{FO17}.

Building on~\cite{CLW16}, Chhita, Ferrari and Spohn derived a variational formula which describes the limiting distribution for random initial conditions which scale to a Brownian motion with the result~\cite{CFS16}
\begin{equation}\label{defFsigma}
F^{(\sigma)}(s)=\P\left(\sup_{t\in\R}\left\{\sqrt2\sigma B(t)+\A_2(t)-t^2\right\}\le s\right)
\end{equation}
where $B$ is a standard two-sided Brownian motion independent of the Airy$_2$ process $\A_2$.
This distribution has two special cases which could be analyzed using exact formulas, namely $\sigma=0$ is the flat case and it reduces to \eqref{eq1} whereas $\sigma=1$ corresponds to the stationary initial condition for the model, so that $F^{(1)}(s)$ is the Baik--Rains distribution~\cite{BR00}.

The characterization through a variational formula is tightly related to the question of universality.
In the framework of this paper, the key universal ingredient is the Airy$_2$ process which is a projection of a more general space-time random process. The study of this process started with the discovery of the KPZ fixed point by Matetski, Quastel and Remenik~\cite{MQR17}, for further properties see~\cite{Pim17,Pim19,CHH19}, and continued with the desription of the full space-time process called the Airy sheet or also directed landscape by Dauvergne, Ortmann and Vir{\'a}g~\cite{DOV18}, see also~\cite{NQR20}.

Deducing concrete information from a variational formula is however not always an easy task.
For example, given \eqref{defFsigma}, it is not clear what are the tails of the distribution.
They have only been known for a long time in the cases $\sigma=0$ and $\sigma=1$, because these distributions had other representations, see e.g.~\cite{PS00}.
Meerson and Schmidt considered the $F^{(\sigma)}$ distribution in~\cite{MS17} and they deduced the correct right tail behavior by a physically motivated but non-rigorous method.
They found that $\ln(1-F^{(\sigma)}(s))\sim -\frac43\frac1{\sqrt{1+3\sigma^4}}s^{3/2}$ for $s\gg 1$. They also performed large scale simulations on the exclusion process confirming their finding.

In this paper we give a rigorous proof of the asymptotics and extend the results of~\cite{MS17} by obtaining upper and lower bound on the prefactor in front of the stretched exponential decay, see Theorem~\ref{thm:Fsigmatail}.
The upper tail distribution is governed by the maximal value of $\sqrt2\sigma B(t)-t^2$. This fact holds already for non-random initial conditions.
For instance, in the case of \eqref{eq1}, the tail behaviour of $1-F_{\rm GOE}(2^{2/3}s)$ matches that of $1-F_{\rm GUE}(s)$ up to the exponential scale as the maximal value of $-t^2$ is obtained at $t=0$. The same was shown for another simple function $h_0$ in \eqref{eq2} as noticed in~\cite{Vai20}.
To make this point explicit, Theorem~\ref{thm:generalcurve} gives the tail decay for a generic non-random initial condition.
One important ingredient for the proof of Theorems~\ref{thm:Fsigmatail} and~\ref{thm:generalcurve} is the observation that, for all $\delta>0$, the tail distribution of
\begin{equation}
\P\left(\sup_{t\in\R}\left(\A_2(t)-\delta t^2\right)>s\right)
\end{equation}
is, in the exponential scale, independent of $\delta$, see Theorem~\ref{thm:Airyparabola} for a detailed statement.

For a given realization of the Brownian motion, if the supremum in \eqref{eq1} is taken over a finite interval instead of $\R$, the distribution we are considering also has a Fredholm determinant expression with a kernel depending on the Brownian motion~\cite{CQR11}.
This representation is however not directly applicable when taking the limit as the finite interval approaches $\R$.
For this purpose, it is better to use the kernel given in terms of hitting times noticed first in~\cite{QR16}.
That representation works well provided that the function $h_0$ is lower than a parabola with a prefactor $3/4$, while in our application we need to get close to $1$.
The explicit kernel representation has some intrinsic technical issues as confirmed also in the simplest case of a hat-shaped $h_0$~\cite{Vai20}.
Our method is mainly probabilistic and avoids the computation of a correlation kernel and its asymptotic analysis.

The second issue that we had to deal with was that the density of the maximum of a Brownian motion with parabolic drift studied first by Groeneboom~\cite{Gro89}, see also~\cite{Gro10} for explicit formulas, contains a term with a linear combination of Airy $\Ai$ and Airy $\Bi$ functions.
The leading term is however coming from a subtle cancellation and it does not follow from the naive asymptotic of the Airy functions.
Fortunately, we could avoid this issue by using an integral representation discovered by Janson, Louchard and Martin-L\"of in~\cite{JLML10}, which we carefully analyzed asymptotically, see Proposition~\ref{prop:Groeneboom}.

The paper is organized as follows.
We state the main results in the rest of the introduction.
We first prove Theorem~\ref{thm:Airyparabola} on the upper tail of the supremum of the Airy$_2$ process minus a parabola with arbitrary coefficient in Section~\ref{s:airy}.
Then Section~\ref{s:groeneboom} is about the asymptotic of the supremum of the Brownian motion minus a parabola.
Section~\ref{s:Fsigmatail} proves Theorem~\ref{thm:Fsigmatail} on the right tail of the limiting distribution $F^{(\sigma)}$ for Brownian initial conditions.
The proof of Theorem~\ref{thm:generalcurve} about the case of general deterministic initial conditions is given in Section~\ref{s:deterministic}.

\paragraph{Acknowledgments.} The work of P.L.~Ferrari was partly funded by the Deutsche Forschungsgemeinschaft (DFG, German Research Foundation) under Germany’s Excellence Strategy -- GZ 2047/1, projekt-id 390685813 and by the Deutsche Forschungsgemeinschaft (DFG, German Research Foundation) -- Projektnummer 211504053 -- SFB 1060.
The work of B.~Vet\H o was supported by the NKFI (National Research, Development and Innovation Office) grants PD123994 and FK123962, by the Bolyai Research Scholarship of the Hungarian Academy of Sciences and by the \'UNKP--20--5 New National Excellence Program of the Ministry for Innovation and Technology from the source of the National Research, Development and Innovation Fund.

\paragraph{Main results}

\begin{theorem}\label{thm:Fsigmatail}
Let $\sigma>0$ be fixed. For $s$ large enough, the right tail of $F^{(\sigma)}(s)$ satisfies
\begin{equation}\label{Fsigmatail}
C_1\, s^{-3/4} e^{-\frac43\frac1{\sqrt{1+3\sigma^4}}s^{3/2}}\leq 1-F^{(\sigma)}(s)\leq C_2 \, s^{3/4} \ln(s)\, e^{-\frac43\frac1{\sqrt{1+3\sigma^4}}s^{3/2}}
\end{equation}
for some constants $C_1,C_2$ independent of $s$.
\end{theorem}
The $s^{-3/4}$ behavior of the prefactor of the lower bound seems to be the correct one. Indeed, for $\sigma=0$ the prefactor is $s^{-3/4}/(4\sqrt{2\pi})$ (see \eqref{eq1.8} below) and for $\sigma=1$ it is given by\footnote{For $\sigma=0$, the prefactor is obtained from equations (1) and (25), (26) of~\cite{BBdF08}, except that in (26) there is a typo, namely $x^{-3/2}$ should be $x^{-3/4}$. It can be also easily obtained from the Fredholm determinant representation of $F_{\rm GOE}$ in~\cite{FS05b}. For $\sigma=1$ the distribution is the Baik--Rains distribution, given in Definition 2 of~\cite{BR00}. The prefactor easily follows using (2.3), (2.6) of\cite{BR00}, as well as (26) of~\cite{BBdF08}.} $s^{-3/4}/\sqrt{\pi}$. In Proposition~\ref{prop:Fsigmalower} and~\ref{prop:Fsigmaupper} we give expressions of the $\sigma$-dependence of $C_1,C_2$.

As it could be expected by the definition of the $F^{(\sigma)}$ distribution \eqref{defFsigma},
the upper tail behaviour of the $F^{(\sigma)}$ distribution is related to tail of the maximum of the Airy$_2$ process minus a parabola.
The variational formula \eqref{eq1} was proved in~\cite{Jo03b}.
In the next result we generalize \cite{Jo03b} in the sense that we compute the tail decay for the supremum with parabola which can have any coefficient between $0$ and $1$.
\begin{theorem}\label{thm:Airyparabola}
Let $(\A_2(t))_{t\in\R}$ denote the Airy$_2$ process.
There is a constant $C>0$ such that for all $c\in(0,1)$
\begin{equation}\label{Airyparabola}
1-F_{\rm GUE}(s)\leq \P\left(\sup_{t\in\R}\left(\A_2(t)-(1-c)t^2\right)>s\right)
\le C\frac{\ln(s/(1-c))}{s^{3/4}\sqrt{1-c}}e^{-\frac43s^{3/2}}
\end{equation}
holds as $s\to\infty$ where the constant $C$ is independent of $s$.
\end{theorem}
This result might be compared to what has been proven for the KPZ equation with narrow wedge initial condition at finite but large time.
In that case, it is known that the decay is exponential with $s^{3/2}$ power, but without a more precise information on the coefficient, see Proposition~4.2 of~\cite{CGH19}.

The lower bound in \eqref{Airyparabola} is obvious by taking $t=0$ instead of the supremum.
Its asymptotic expansion follows from (1) and (25) of~\cite{BBdF08}, namely
\begin{equation}\label{eqGUEtail}
1-F_{\rm GUE}(s)=\frac{1}{16\pi s^{3/2}}e^{-\frac43 s^{3/2}}(1+\O(s^{-3/2}))
\end{equation}
as $s\to\infty$.
We remark that for $c\leq 0$ the upper bound in \eqref{Airyparabola} is trivial as well, since
\begin{equation}\label{eq1.8}
\P\left(\sup_{t\in\R}\left(\A_2(t)-(1-c)t^2\right)>s\right) \leq \P\left(\sup_{t\in\R}\left(\A_2(t)-t^2\right)>s\right)\sim \frac{1}{4\sqrt{2\pi} s^{3/4}}e^{-\frac43 s^{3/2}}
\end{equation}
where we used \eqref{eq1} and the $x\to\infty$ tail asymptotic
\begin{equation}\label{eqGOEtail}
1-F_{\rm GOE}(x)\sim \frac{e^{-\frac23x^{3/2}}}{4\sqrt\pi x^{3/4}}.
\end{equation}

\begin{theorem}\label{thm:generalcurve}
Let $h_0:\R\to\R$ be a function satisfying $h_0(t)\leq A+(1-\varepsilon) t^2$ for all $t\in\R$, for some constants $A\in\R$ and $\varepsilon>0$.
Let $\kappa(h_0)=\sup_{t\in\R} \{h_0(t)-t^2\}$ and let $M>0$ be large enough so that $h_0(t)\leq \kappa(h_0)+(1-\frac\varepsilon2) t^2$ for all $|t|\geq M$.
Then there are positive real constants $C_1$ and $C_2$ which do not depend on the function $h_0$ and $s$, such that for $s$ large enough
\begin{equation}\label{generalcurvebound}
C_1 \frac{e^{-\frac43 (s-\kappa(h_0))^{3/2}}}{(s-\kappa(h_0))^{3/2}}\leq \P\left(\sup_{t\in\R}\left\{h_0(t)+\A_2(t)-t^2\right\}\geq s\right)\leq C_2 M \frac{e^{-\frac43 (s-\kappa(h_0))^{3/2}}}{(s-\kappa(h_0))^{1/4}}.
\end{equation}
\end{theorem}

\section{Supremum of the Airy$_2$ process minus a parabola}\label{s:airy}

The aim of this section is to prove Theorem~\ref{thm:Airyparabola} about the upper tail behaviour of the Airy$_2$ process minus a parabola with arbitrary coefficient. The first ingredient is a simple bound on the supremum of the Airy$_2$ process over a finite interval.

\begin{lemma}\label{lemma:Airyfinite}
There is a explicit constant $C$ such that for all $a>0$
\begin{equation}\label{Airyfinite}
\P\bigg(\sup_{t\in[0,a]}\A_2(t)>s\bigg)\le C\,\frac{e^{-\frac43(s-a^2)^{3/2}}}{(s-a^2)^{3/4}}
\end{equation}
holds if $s$ is large enough. From this we get the bound
\begin{equation}\label{Airyfinite2}
\P\bigg(\sup_{t\in[0,a]}\A_2(t)>s\bigg)\le C'\,\frac{a}{s^{1/4}}e^{-\frac43 s^{3/2}}
\end{equation}
for large $s$ with some other constant $C'$.
\end{lemma}

\begin{proof}
The probability on the left-hand side of \eqref{Airyfinite} can be upper bounded as
\begin{equation}\label{Airyfinitebound}\begin{aligned}
\P\bigg(\sup_{t\in[0,a]}\A_2(t)>s\bigg)&\le\P\bigg(\sup_{t\in[0,a]}\left(\A_2(t)-t^2\right)>s-a^2\bigg)\\
&\le\P\left(\sup_{t\in\R}\left(\A_2(t)-t^2\right)>s-a^2\right)\\
&=1-F_{\text{GOE}}\left(2^{2/3}\left(s-a^2\right)\right)
\end{aligned}\end{equation}
where we extended the range of $t$ in the second inequality and used \eqref{eq1} in the last equality above which is the result of~\cite{Jo03b}.
The proof of the first inequality \eqref{Airyfinite} with $C=1/(4\sqrt{2\pi})$ is completed by applying the upper tail behaviour of the $F_{\rm GOE}$ distribution \eqref{eqGOEtail}.

For \eqref{Airyfinite2}, we divide the interval $[0,a]$ into pieces of length $1/\sqrt{s}$ and by the union bound we get
\begin{equation}\label{eq2.3}\begin{aligned}
\P\bigg(\sup_{t\in[0,a]}\A_2(t)>s\bigg) &\le\sum_{k=1}^{a\sqrt{s}}\P\bigg(\sup_{t\in[(k-1)/\sqrt{s},k/\sqrt{s}]}\A_2(t)>s\bigg)\\
&\le C\, a \sqrt{s} \frac{e^{-\frac43 (s-1/s)^{3/2}}}{(s-1/s)^{3/4}}\\
&\leq C'\, \frac{a}{s^{1/4}}e^{-\frac43 s^{3/2}}
\end{aligned}\end{equation}
for some constant $C'$ where we applied stationarity of the Airy$_2$ process and \eqref{Airyfinite} in the second inequality and the bound $-(s-1/s)^{3/2}\leq -s^{3/2}+2$ for all $s\geq 1$ in the third inequality above.
\end{proof}

With this lemma we are ready to prove Theorem~\ref{thm:Airyparabola}.
\begin{proof}[Proof of Theorem~\ref{thm:Airyparabola}]
To bound the probability that the Airy$_2$ process remains below a parabola, consider an increasing sequence $x_0=0<x_1<x_2<\dots$ to be a partition of $\R_+$ to be specified later.
Then we have by symmetry and the union bound
\begin{equation}\label{Airyparabolabound}\begin{aligned}
\P\left(\sup_{t\in\R}\left(\A_2(t)-(1-c)t^2\right)>s\right)
&\le2\P\left(\cup_{k=0}^\infty\left\{\exists t\in[x_k,x_{k+1}]:\A_2(t)>s+(1-c)t^2\right\}\right)\\
&\le2\sum_{k=0}^\infty\P\left(\exists t\in[x_k,x_{k+1}]:\A_2(t)>s+(1-c)t^2\right)\\
&\le2\sum_{k=0}^\infty\P\left(\exists t\in[x_k,x_{k+1}]:\A_2(t)>s+(1-c)x_k^2\right)\\
&\le2C\sum_{k=0}^\infty\frac{e^{-\frac43\left(s+(1-c)x_k^2-(x_{k+1}-x_k)^2\right)^{3/2}}}
{\left(s+(1-c)x_k^2-(x_{k+1}-x_k)^2\right)^{3/4}}
\end{aligned}\end{equation}
where we decreased the barrier which $\A_2(t)$ has to reach in $[x_k,x_{k+1}]$ in the third inequality, while in the forth inequality we used Lemma~\ref{lemma:Airyfinite} and the translation invariance of the Airy$_2$ process.

The $k=0$ term in the sum on the right-hand side of \eqref{Airyparabolabound} is $e^{-\frac43(s-x_1^2)^{3/2}}/(s-x_1^2)^{3/4}$, hence if we choose $x_1=1/\sqrt s$, then it is still bounded by a constant multiplied by $e^{-\frac43s^{3/2}}/s^{3/4}$.
The rest of the sequence is chosen so that it satisfies
\begin{equation}\label{recurrence}
(x_{k+1}-x_k)^2=\frac14(1-c)x_k^2
\end{equation}
for $k=1,2,\dots$ which is the geometric choice $x_{k+1}=\gamma x_k$ with $\gamma=1+\frac{\sqrt{1-c}}2$.
With this sequence, the right-hand side of \eqref{Airyparabolabound} can be bounded as
\begin{equation}\label{Airyparabolabound2}\begin{aligned}
\P\left(\sup_{t\in\R}\left(\A_2(t)-(1-c)t^2\right)>s\right)
&\le2C\frac{e^{-\frac43(s-\frac1{s})^{3/2}}}{(s-\frac1s)^{3/4}}+2C\sum_{k=1}^\infty \frac{e^{-\frac43\left(s+\frac34(1-c)x_k^2\right)^{3/2}}}{s^{3/4}}\\
&\le2C'\frac{e^{-\frac43s^{3/2}}}{s^{3/4}}+2C'\frac{e^{-\frac43s^{3/2}}}{s^{3/4}}
\sum_{k=1}^\infty e^{-\frac43\left(\frac34(1-c)\gamma^{2(k-1)} s^{-1}\right)^{3/2}}
\end{aligned}\end{equation}
where we used the inequality $(a+b)^{3/2}\ge a^{3/2}+b^{3/2}$ in the last step.
The sum on the right-hand side of \eqref{Airyparabolabound2} can be upper bounded using Lemma~\ref{lem2.3} below with $\alpha=\gamma^3$ and $\beta=\sqrt{3}(1-c)^{3/2}/(2s^{3/2})$ as
\begin{equation}\label{sumbound}
\sum_{k=1}^\infty e^{-\frac{\sqrt{3}(1-c)^{3/2}}{2s^{3/2}}\gamma^{3(k-1)}}
\le\frac{\ln\left(1+\frac{2\gamma^3s^{3/2}}{\sqrt{3}(1-c)^{3/2}}\right)}{3\ln\gamma}
\le\wt C\,\frac{\ln(s/(1-c))}{\sqrt{1-c}}
\end{equation}
for $s$ large enough with some $\wt C$  which does not depend on $c$.
In the last inequality above we used that $\ln\gamma\sim \frac12\sqrt{1-c}$ as $c\to1$.
The inequalities \eqref{Airyparabolabound2} and \eqref{sumbound} together prove \eqref{Airyparabola}.
\end{proof}

\begin{lemma}\label{lem2.3}
Let $\alpha>1$ and $\beta>0$. Then
\begin{equation}
\sum_{k=0}^\infty e^{-\beta\alpha^k}\le \frac{\ln(1+\alpha/\beta)}{\ln\alpha}
\end{equation}
\end{lemma}

\begin{proof}
We start by bounding the sum by an integral as
\begin{equation}\label{eq2.9}
\sum_{k=0}^\infty e^{-\beta\alpha^k}\leq \int_{-1}^\infty \d x\, e^{-\beta \alpha^x}.
\end{equation}
The change of variables $w=\beta\alpha^x$ gives $\frac{\d w}{\d x}=w \ln\alpha$ so that the right-hand side of \eqref{eq2.9}
\begin{equation}
\int_{-1}^\infty \d x\, e^{-\beta \alpha^x}= \frac{1}{\ln\alpha} \int_{\beta/\alpha}^\infty \d w\,\frac{e^{-w}}{w}
=:\frac{1}{\ln\alpha}E_1(\beta/\alpha).
\end{equation}
Equation (5.1.20) of~\cite{AS84} provides a bound on the exponential integral function $E_1$, namely $E_1(x)\leq e^{-x}\ln(1+1/x)$ for $x>0$. Thus
\begin{equation}
\sum_{k=0}^\infty e^{-\beta\alpha^k} \leq \frac{1}{\ln\alpha} e^{-\beta/\alpha}\ln(1+\alpha/\beta)
\end{equation}
which gives the claimed bound since $\alpha,\beta>0$.
\end{proof}

\section{Supremum of Brownian motion minus a parabola}\label{s:groeneboom}

\begin{proposition}\label{prop:Groeneboom}
Let $G(x)=\P\left(\max_{t\in\R}(B(t)-\tfrac12 t^2)\geq x\right)$ where $B(t)$ is a standard two-sided Brownian motion.
Then, as $x\to\infty$ we have
\begin{equation}\label{eqGroenExpansion}
G(x)= 3^{-1/2} e^{-\frac43 \sqrt{\frac23}\,x^{3/2}} (1+\O(x^{-1/4}))
\end{equation}
as well as
\begin{equation}\label{eqGroenExpansionB}
-\frac{\d}{\d x}G(x)= \frac{2\sqrt{2}}{3} e^{-\frac43 \sqrt{\frac23}\,x^{3/2}} \sqrt{x} (1+\O(x^{-1/4})).
\end{equation}
Consequently, for any $c>0$ the density of the random variable $\max_{t\in\R}(B(t)-ct^2)$ satisfies
\begin{equation}\label{eqGroenExpansionC}
f_c(x):=-\frac{\d}{\d x}G\left((2c)^{1/3}x\right)=\frac43\sqrt c\,e^{-\frac43 \sqrt{\frac43c}\,x^{3/2}} \sqrt{x} (1+\O(x^{-1/4}))
\end{equation}
as $x\to\infty$.
\end{proposition}
The proof is given below. The distribution function $G(x)$ is written as a contour integral in Lemma~3.5 of~\cite{JLML10} as follows:
\begin{equation}\label{Gcontint}
G(x)=\frac{1}{2\I}\int_\gamma\d z\, \frac{\Hi(z)}{\Ai(z)}\Ai(z+2^{1/3}x)
\end{equation}
where $\gamma$ is a path passing to the right of all zeroes of the Airy function $\Ai$ from $-\I\infty$ to $\I\infty$.
The function $\Hi$ is defined by (see (10.4.44) of~\cite{AS84})
\begin{equation}\label{defHi}
\Hi(z)=\pi^{-1}\int_0^\infty \d t\, e^{-t^3/3+z t}.
\end{equation}
In~\cite{JLML10}, the contour $\gamma$ in \eqref{Gcontint} was chosen to come from $e^{-i\theta}\infty$ and arrive to $e^{i\theta}\infty$ with $\theta$ slightly larger than $\pi/2$.
The reason why this contour can be deformed to the vertical one is the following.
Lemma~A.1 of~\cite{JLML10} was used to argue for convergence of the integral \eqref{Gcontint}: they showed the decay of the ratio of the two Airy $\Ai$ functions, together with a bound on $\Hi$ for contours with angle more than $\pi/2$.
The bound in Lemma~\ref{lemma:Hibound} is enough to get the convergence also for vertical contours.

For the proof of \eqref{prop:Groeneboom}, the asymptotic of the Airy and $\Hi$ functions will be needed.
By (10.4.90) of~\cite{AS84}, for $x$ real,
\begin{equation}\label{eq3.6}
\Hi(x)\sim \pi^{-1/2}x^{-1/4}e^{\frac23 x^{3/2}}\textrm{ as }x\to\infty.
\end{equation}
For our purposes, we need also the asymptotic behavior of $\Hi(z)$ for complex-valued $z$ close to $2^{1/3} x/3$ which we state below and prove later in this section.
\begin{lemma}\label{LemmaHi}
Let $z$ be such that $|\arg(z)|<\pi/3$. Then for large $z$ we have the asymptotic behavior
\begin{equation}\label{Hiasymptotic}
\Hi(z)\, e^{-\frac23 z^{3/2}} = \pi^{-1/2} z^{-1/4}+\O(|z|^{-1/2}).
\end{equation}
\end{lemma}

\begin{lemma}\label{lemma:Hibound}
Let $\theta\in[\pi/2,3\pi/2]$ and $x\in\R$.
Then for all $y\geq 0$,
\begin{equation}\label{Hibound}
\left|\Hi\left(x+e^{\I\theta}y\right)\right|\le\Hi(x).
\end{equation}
\end{lemma}

\begin{proof}
By the defintion \eqref{defHi},
\begin{equation}\label{Hibound2}
\Hi\left(x+e^{\I\theta}y\right)=\pi^{-1}\int_0^\infty \d t\, e^{-t^3/3}e^{xt}e^{e^{\I\theta}y t}
\end{equation}
where $|e^{e^{\I\theta}y t}|\le1$ for all $\theta\in[\pi/2,3\pi/2]$ and $t\ge0$.
Since $e^{-t^3/3}e^{xt}$ is positive, the absolute value of the integral in \eqref{Hibound2} can be upper bounded by the integral of $e^{-t^3/3}e^{xt}$ which yields \eqref{Hibound}.
\end{proof}

\begin{proof}[Proof of Proposition~\ref{prop:Groeneboom}]
Let us first prove \eqref{eqGroenExpansion}. In order to estimate $G(x)$ for large $x$, we use the integral representation \eqref{Gcontint}.
The integration contour is chosen to be vertical $\gamma=\frac{2^{1/3}}{3}x+\I\R$.
If $x>0$, then for $z\in\gamma$, we have $\arg(z)\in[-\pi/2,\pi/2]$.
Hence we can use the asypmtotics of the Airy function
\begin{equation}\label{Airyasymp}
\Ai(z)=\tfrac12 \pi^{-1/2} z^{-1/4} e^{-\frac23 z^{3/2}}(1+O(1/z)) \textrm{ for }|\arg(z)|<\pi,
\end{equation}
see (10.4.59) of~\cite{AS84}. Let us parameterize the path $\gamma$ as $z=\frac{2^{1/3}}3x+\I 2^{1/3} x v$, $v\in\R$.

\emph{Contribution for $|v|>1/3$.}
Using Lemma~\ref{lemma:Hibound} and~\ref{LemmaHi} for the $\Hi$ function and the asymptotic \eqref{Airyasymp} on the Airy functions, the contribution for $|v|>1/3$ is bounded by
\begin{equation}\label{eq3.11}
C x^{3/4}\int_{1/3}^\infty \d v\, e^{-x^{3/2}g(v)},\quad g(v)=\frac{2\sqrt{2}}{9\sqrt{3}}\Re[(4+3 \I v)^{3/2}-(1+3 \I v)^{3/2}-1]
\end{equation}
for some constant $C$. Notice that for $a\geq 0$,
\begin{equation}
\frac{\d \Re((a+\I v)^{3/2})}{\d a}=-\frac32 \Im\sqrt{a+\I v}
\end{equation}
which is an increasing function of $a$.
This implies that $g(v)$ is monotone increasing in $v\geq 0$ with $g(v)\sim\sqrt{v}$ as $v\to\infty$.
We get that the leading behavior of the integral is bounded by $C' e^{-x^{3/2} g(1/3)}$ with $g(1/3)-g(0)<0$, for some other constant $C'$. Thus this contribution is vanishing with respect to the leading one computed below.

\emph{Contribution for $|v|\leq 1/3$.}
For the leading contribution, we use the asymptotic expansion of Lemma~\ref{LemmaHi} and \eqref{Airyasymp}, with the result that the contribution of $G(x)$ coming from the integral over $|v|\leq 1/3$ is given by
\begin{multline}
2^{-2/3}x\int_{|v|\leq 1/3} \d v\,
\frac{\Hi\left(\frac{2^{1/3}}3x+\I2^{1/3}xv\right)}{\Ai\left(\frac{2^{1/3}}3x+\I2^{1/3}xv\right)}
\Ai\left(\frac{4\cdot2^{1/3}}3x+\I2^{1/3}xv\right)\\
=\frac{x^{3/4}}{\sqrt{\pi}2^{3/4}} \int_{|v|\leq 1/3} \d v\,\frac{e^{ x^{3/2} h(v)}}{(4/3+\I v)^{1/4}}(1+\O(1/x^{1/4}))
\end{multline}
with $h(v)=\frac43 \sqrt{2} ((1/3+\I v)^{3/2}-(4/3+\I v)^{3/2}/2)$. Further, one notices that $\Re(h(v))$ is strictly increasing for $v<0$ and strictly decreasing for $v>0$ with a quadratic approximation for small $v$ given by
\begin{equation}
h(v)= -\frac{4\sqrt{2}}{3\sqrt{3}}-\frac{3\sqrt{3}}{4\sqrt{2}}v^2+\O(v^3).
\end{equation}
Then, using standard steep descent analysis, we get \eqref{eqGroenExpansion}.

The proof of \eqref{eqGroenExpansionB} is similar. The only difference is that we need to replace the asymptotic expansion of $\Ai(z+2^{1/3}x)$ with the one of $2^{1/3}\Ai'(z+2^{1/3}x)$. By (10.4.61) of~\cite{AS84}, we have $\Ai'(z)=-\frac12\pi^{-1/2}z^{1/4}e^{-2 z^{3/2}/3}(1+\O(1/z))$. Thus in the asymptotic analysis we need to replace $z^{-1/4}$ with $-2^{1/3}z^{1/4}$. This gives the claimed result.

By replacing $t$ by $(2c)^{-2/3}t$ and using Brownian rescaling,
\begin{equation}
\P\left(\max_{t\in\R}\left(B(t)-c t^2\right)\geq x\right)
=\P\left(\max_{t\in\R}\left(B(t)-\tfrac12 t^2\right)\geq x(2c)^{1/3}\right)
=G\left((2c)^{1/3}x\right),
\end{equation}
hence \eqref{eqGroenExpansionC} follows from \eqref{eqGroenExpansionB} by substitution.
\end{proof}

Now we prove the claimed asymptotic expansion of $\Hi$.
\begin{proof}[Proof of Lemma~\ref{LemmaHi}]
By symmetry with respect to the real axis ($z\mapsto \bar z$), consider $z$ with $\arg(z)\in [0,\pi/3)$. We parameterize $z=r e^{\I \theta}$ with $r>0$ and $\theta\in [0,\pi/3)$. We introduce the change of variables $t=\sqrt{z}+u$ which yields
\begin{equation}
-\frac13 t^3+z t = \frac23 z^{3/2}-\frac13 u^3-\sqrt{z} u^2
\end{equation}
and we get
\begin{equation}\label{eqHiRepr}
\Hi(z)\, e^{-2 z^{3/2}/3} = \pi^{-1} \int_{-\sqrt{z}}^\infty \d u\, e^{-u^3/3-\sqrt{z} u^2}.
\end{equation}
For large values of $|z|$, the leading contribution comes from a neighborhood of $0$.
Consider the integration contour $\Gamma=\Gamma_1\vee\Gamma_2$ where $\Gamma_1=\{-\sqrt{z}+\I y,0\leq y\leq \Im(\sqrt{z})\}$ and $\Gamma_2=\{x, -\Re(\sqrt{z})\leq x<\infty\}$.

Consider first the contribution on the contour $\Gamma_1$. Let $f(u)=\Re(-u^3/3-\sqrt{z} u^2)$.
Then we have
\begin{equation}
f(-\sqrt{z})=-\frac23 \Re(z^{3/2})=-\frac23 r^{3/2}\cos(3\theta/2)<0 = f(0).
\end{equation}
Setting $u=-\sqrt{z}+\I y=-\sqrt{r}\cos(\theta/2)-\I\sqrt{r}\sin(\theta/2)+\I y$, we obtain
\begin{equation}
\Re(-u^3/3-\sqrt{z}u^2) = {\rm const} -\sqrt{r}\sin(\theta) y,
\end{equation}
which is decreasing in $y$.
Thus the contribution of the integral of \eqref{eqHiRepr} over $\Gamma_1$ can bounded by the maximum of the integrand times the length of the contour, that is, by $\pi^{-1} \Im(\sqrt{z}) e^{f(-\sqrt{z})}\leq C e^{-\frac23 r^{3/2}\cos(3\theta/2)}$.

Next we focus on the contribution over $\Gamma_2$. We have, for all $u\geq -\sqrt{r}\cos(\theta/2)=-\Re(\sqrt z)$, the bound
\begin{equation}
\Re(-u^3/3-\sqrt{z}u^2)=-\tfrac13u^3-\sqrt{r}\cos(\theta/2) u^2\leq -\tfrac23 \sqrt{r}\cos(\theta/2) u^2.
\end{equation}
For any $\delta>0$, which can be chosen later as a function of $r$, the tails of the Gaussian integral gives
\begin{equation}
\left|\pi^{-1} \int_{\Gamma_2\setminus \{|u|\leq \delta\}} du\, e^{-u^3/3-\sqrt{z} u^2}\right|\leq C e^{-\tfrac23 \sqrt{r}\cos(\theta/2) \delta^2}.
\end{equation}
Next, the local contribution is close to the integral with only the quadratic term. Indeed, using $|e^x-1|\leq |x| e^{|x|}$, we have
\begin{multline}
\left|\pi^{-1} \int_{\{|u|\leq \delta\}} \d u\, e^{-u^3/3-\sqrt{z} u^2}-\pi^{-1} \int_{\{|u|\leq \delta\}} \d u\, e^{-\sqrt{z} u^2}\right|\\
\leq  \pi^{-1} \int_{\{|u|\leq \delta\}} \d u\left|e^{-u^3/3-\sqrt{z} u^2}\right| \frac{|u|^3}{3}=\O(\delta^3).
\end{multline}
Extending the integration in the Gaussian integral to $\R$, we only make a small error, namely
\begin{equation}
\left|\pi^{-1} \int_{\{|u|\leq \delta\}} \d u\, e^{-\sqrt{z} u^2} - \pi^{-1} \int_{\R} \d u\, e^{-\sqrt{z} u^2}\right|\leq \O(e^{-\delta^2 \sqrt{r}\cos(\theta/2)}).
\end{equation}
Finally, the Gaussian integral can be computed explicitly as
\begin{equation}
\pi^{-1} \int_{\R} \d u\, e^{-\sqrt{z} u^2} = \pi^{-1/2} z^{-1/4}.
\end{equation}
Combining all these bounds, we get
\begin{equation}
\Hi(z)\, e^{-\frac23 z^{3/2}}=\pi^{-1/2} z^{-1/4}
+\O(\delta^3,\sqrt{r} e^{-\frac23r^{3/2}\cos(3\theta/2)}, e^{-\frac23 \sqrt{r}\cos(\theta/2)\delta^2}).
\end{equation}
Now, since $\theta\in [0,\pi/3)$, we have $\cos(3\theta/2),\cos(\theta/2)\in [1/\sqrt{2},1]$.
By choosing $\delta=r^{-1/6}$, we get \eqref{Hiasymptotic}.
\end{proof}

\section{Tail bounds for random initial conditions}\label{s:Fsigmatail}

In this section we prove Theorem~\ref{thm:Fsigmatail} about the upper tail decay of the $F^{(\sigma)}(s)$ distribution, which follows by combining Propositions~\ref{prop:Fsigmalower} and~\ref{prop:Fsigmaupper} below.

Fix $c\in(0,1)$ and let
\begin{equation}\label{deftauM}
\tau_c=\arg\max\left(\sqrt2\sigma B(t)-ct^2\right),\quad
M_c=\max_{t\in\R}\left(\sqrt2\sigma B(t)-ct^2\right)=\sqrt2\sigma B(\tau_c)-c\tau_c^2
\end{equation}
be the position and the value of the maximum of the two-sided Brownian motion $\sqrt2\sigma B(t)$ with diffusion coefficient $2\sigma^2$ where $B(t)$ is a standard one.
Note that
\begin{equation}
M_c=\max_{t\in\R}\left(\sqrt2\sigma B(t)-ct^2\right)
=\max_{u\in\R}\left(\sqrt2\sigma B\left(\frac u{2\sigma^2}\right)-\frac{cu^2}{4\sigma^4}\right)
\stackrel{\d}{=}\max_{u\in\R}\left(B(u)-\frac c{4\sigma^4}u^2\right)
\end{equation}
where the second equality follows by the change of variables $t=u/(2\sigma^2)$ and the third one by Brownian scaling.
As a consequence, the random variable $M_c$ has density $f_{\frac c{4\sigma^4}}(x)$ where $f_c$ was defined in \eqref{eqGroenExpansionC}.

\paragraph{Lower bound.} The idea of the lower bound is to use the inequality
\begin{equation}
\sup_{t\in \R} \big\{\sqrt{2}\sigma B(t)+\A_2(t)-t^2\big\}\geq \sqrt{2}\sigma B(t_0)+\A_2(t_0)-t_0^2
\end{equation}
which holds for any choice of $t_0\in\R$.
Furthermore, $\P\left(\sqrt{2}\sigma B(t_0)+\A_2(t_0)-t_0^2>s\right)$ will be the largest, that is, we get the best lower bound if we take a time $t_0$ where $\sqrt{2}\sigma B(t_0)+\A_2(t_0)-t_0^2$ is the largest.
As the Airy$_2$ process is stationary and independent of $B$, it does not make any difference for $\A_2(t_0)$ which time is chosen. Thus the idea is to choose $t_0$ to be the random time $\tau_1$ which maximizes $\sqrt{2}\sigma B(t)-t^2$.

\begin{proposition}\label{prop:Fsigmalower}
For all $\sigma>0$, there is a constant $C_1$ independent of $s$ such that
\begin{equation}
1-F^{(\sigma)}(s) \geq C_1 \sigma^2 (1+3\sigma^4)^{1/4} s^{-3/4}e^{-\frac43\frac1{\sqrt{1+3\sigma^4}}s^{3/2}}
\end{equation}
holds for $s\gg \max\{\sigma^{-4},\sigma^4\}$.
\end{proposition}
\begin{proof}
The upper tail of the $F^{(\sigma)}$ distribution can be rewritten as
\begin{equation}\label{Fsigmaconditional}
1-F^{(\sigma)}(s)=\E\left(\P\left(\sup_{t\in\R}\left(\sqrt2\sigma B(t)+\A_2(t)-t^2\right)>s\,\Big|\,\tau_1\right)\right)
\end{equation}
by conditioning on the value of the time $\tau_1$.
The conditional probability on the right-hand side of \eqref{Fsigmaconditional} can be lower bounded by replacing the supremum of $\sqrt2\sigma B(t)+\A_2(t)-t^2$ with its value at $t=\tau_1$ to get
\begin{equation}\label{Fsigmalower}\begin{aligned}
1-F^{(\sigma)}&\ge\E\left(\P\left(\sqrt2\sigma B(\tau_1)+\A_2(\tau_1)-\tau_1^2>s\,\big|\,\tau_1\right)\right)\\
&=\E\left(\P\left(\A_2(\tau_1)>s-M_1\,|\,\tau_1\right)\right)\\
&=\E\left(1-F_{\text{GUE}}(s-M_1)\right)
\end{aligned}\end{equation}
where the first equality follows by the definition \eqref{deftauM} of $M_1$ and by rearranging.
In the second equality, we used that the Airy$_2$ process is stationary with GUE Tracy--Widom distribution at any position independently of $\tau_1$.

By using Proposition~\ref{prop:Groeneboom} about the asymptotic of the density of $M_1$ and the tail decay of the GUE Tracy--Widom distribution, see \eqref{eqGUEtail}, one gets that the right-hand side of \eqref{Fsigmalower} can be lower bounded by
\begin{equation}\label{1-FGUEintegral}\begin{aligned}
&\E\left(1-F_{\text{GUE}}(s-M_1)\right)\\
&\qquad\ge\frac43\sqrt{\frac1{4\sigma^4}}\frac1{16\pi}\int_0^s\d m\,e^{-\frac43\sqrt{\frac43\frac1{4\sigma^4}}m^{3/2}}
e^{-\frac43(s-m)^{3/2}}\frac{\sqrt{m}}{(s-m)^{3/2}}\left(1+R(s,m))\right)\\
&\qquad=\frac1{24\pi\sigma^2}\int_0^1\d \mu\,e^{-\frac43\sqrt{\frac43\frac1{4\sigma^4}}s^{3/2}\mu^{3/2}}
e^{-\frac43s^{3/2}(1-\mu)^{3/2}}\frac{\sqrt{\mu}}{(1-\mu)^{3/2}}\left(1+R(s,s \mu)\right)
\end{aligned}\end{equation}
with the change of variables $m=s\mu$. Here $R(s,m)=\O((s-m)^{-3/2},m^{-1/4})$ and $R(s,s\mu)=\O(s^{-3/2}(1-\mu)^{-3/2},s^{-1/4}\mu^{-1/4})$ is meant as $s\to\infty$ with $\mu\in (0,1)$.

Let $g(\mu)=-\frac43\sqrt{\frac43\frac1{4\sigma^4}}\mu^{3/2}-\frac43(1-\mu)^{3/2}$. One can compute that
\begin{equation}\label{mu0}
g'(\mu)=0 \textrm{ for }\mu=\mu_0=\frac{3\sigma^4}{1+3\sigma^4}
\end{equation}
as well as
\begin{equation}
g''(\mu)<0\text{ for all }\mu\in [0,1].
\end{equation}
In particular, Taylor expansion gives $g(\mu)=g(\mu_0)-\alpha (\mu-\mu_0)^2+\O((\mu-\mu_0)^3)$ with $g(\mu_0)=-\frac43\frac1{\sqrt{1+3\sigma^4}}$ and $\alpha=(1+3\sigma^4)^{3/2}/(6\sigma^4)$.

The main contribution of the integral on the right-hand side of \eqref{1-FGUEintegral} comes from the regime $|\mu-\mu_0|\sim\frac1{s^{3/4}\sqrt{g''(\mu_0)}}=\frac{\sqrt3\sigma^2}{(1+3\sigma^4)^{3/4}}\frac1{s^{3/4}}$.
We assume that $s\gg \max\{\sigma^4,\sigma^{-4}\}$ which can be written equivalently as $s^{-1/4}\ll\sigma\ll s^{1/4}$.
Next we show that the error terms in $R$ in \eqref{1-FGUEintegral} are small in the regime of the main contribution.
If $\sigma\to0$ as $s\to\infty$, then $\mu_0\sim3\sigma^4$ by \eqref{mu0} and for the regime which we consider $|\mu-\mu_0|\sim\frac{\sqrt3\sigma^2}{s^{3/4}}=o(\sigma^4)$ holds as long as $\sigma\gg s^{-1/4}$.
Hence $\mu s\to\infty$ and $R\to0$ in the regime of the main contribution.
If $\sigma\to\infty$ with $s\to\infty$, then $1-\mu_0\sim\frac1{3\sigma^4}$ and the width of the regime considered is $\sim\frac1{3^{3/4}\sigma^3s^{3/4}}=o(\sigma^{-4})$ provided that $\sigma\ll s^{1/4}$.
Furthermore, $(1-\mu)s\to\infty$ and $R\to0$ in the regime which gives the main contribution.
The error $R$ also goes to $0$ in the regime above if $\sigma$ remains bounded away from $0$ and infinity.

In the regime of $\mu$ that we consider, the higher order terms of the expansion are controlled by the quadratic term for all $s\gg \min\{\sigma^4,\sigma^{-4}\}$.
Thus the quadratic approximation leads to the lower bound
\begin{equation}\begin{aligned}
\E\left(1-F_{\text{GUE}}(s-M_1)\right)&\ge\frac1{24\sqrt\pi\sigma^2}\frac{\sqrt{\mu_0}}{(1-\mu_0)^{3/2}\sqrt\alpha s^{3/4}} e^{-g(\mu_0) s^{3/2}}(1+\O(s^{-1/4}))\\
&=\frac{\sigma^2(1+3\sigma^4)^{1/4}}{4\sqrt{2\pi}}s^{-3/4}e^{-\frac43\frac1{\sqrt{1+3\sigma^4}}s^{3/2}}(1+\O(s^{-1/4})).
\end{aligned}\end{equation}
\end{proof}

\paragraph{Upper bound.}
This strategy for getting the upper bound is different. We noticed that the tail distribution of $\sup_{t\in\R}(\A_2(t)-(1-c)t^2)$ is, in the exponential scale, independent of $c$ provided that $c<1$.
This implies that the tail distribution will be determined mostly by the tail of $M_c=\sup_{t\in\R}(\sqrt{2}\sigma B(t)-c t^2)$.
The proof of the upper bound goes by conditioning on the value of $M_c$ and bounding $\sqrt{2}\sigma B(t)-c t^2$ by $M_c$ from above.
\begin{proposition}\label{prop:Fsigmaupper}
For all $\sigma>0$, there is a constant $C_2$ independent of $s$ such that
\begin{equation}\label{Fsigmaupper}
1-F^{(\sigma)}(s) \leq C_2 \sigma^6(1+3\sigma^4)^{-2} s^{3/4}\ln(s)\,e^{-\frac43\frac1{\sqrt{1+3\sigma^4}}s^{3/2}}
\end{equation}
holds for $s\gg \max\{\sigma^{-4},\sigma^4\}$.
\end{proposition}
\begin{proof}
To get an upper bound, one can write the event
\begin{multline}
\left\{\sup_{t\in\R}\left(\sqrt2\sigma B(t)+\A_2(t)-t^2\right)>s\right\}\\
=\left\{\exists t\in\R:\left(\sqrt2\sigma B(t)-ct^2\right)+\left(\A_2(t)-(1-c)t^2\right)>s\right\}
\end{multline}
for any $c\in(0,1)$.
Since the maximum of the first term on the right-hand side is $M_c$,
it holds for any $t\in\R$ that $\sqrt2\sigma B(t)-ct^2\le M_c$,
and one can bound the upper tail of $F^{(\sigma)}$ as
\begin{equation}\label{Fsigmaupperbound}\begin{aligned}
1-F^{(\sigma)}(s)&\le\P\left(\exists t\in\R:M_c+\left(\A_2(t)-(1-c)t^2\right)>s\right)\\
&=\E\left(\P\left(\sup_{t\in\R}\left(\A_2(t)-(1-c)t^2\right)>s-M_c\,\Big|\,M_c\right)\right)
\end{aligned}\end{equation}
where the last equality follows by conditioning and rearrangement.

Now the upper bound on the right-hand side of \eqref{Fsigmaupperbound} can be bounded by an integral using Proposition~\ref{prop:Groeneboom} about the density of $M_c$ and by Theorem~\ref{thm:Airyparabola}.
Hence we get
\begin{equation}\begin{aligned}
&1-F^{(\sigma)}(s)\\
&\quad\le C\int_0^s\d m\,e^{-\frac43\sqrt{\frac43\frac c{4\sigma^4}}m^{3/2}}e^{-\frac43(s-m)^{3/2}}
\frac{\sqrt{m}\ln((s-m)/(1-c))}{(s-m)^{3/4}\sqrt{1-c}} \left(1+\O(m^{-1/4})\right)\\
&\quad\leq C \int_0^1\d \mu\,e^{-\frac43\sqrt{\frac43\frac c{4\sigma^4}}\mu^{3/2}s^{3/2}}e^{-\frac43(1-\mu)^{3/2}s^{3/2}}\,
\frac{s^{3/4}\sqrt{\mu}\ln((s(1-\mu))/(1-c))}{(1-\mu)^{3/4}\sqrt{1-c}}\\
&\qquad\times\left(1+\O\Big(\frac{1}{(s\mu)^{1/4}}\Big)\right).
\end{aligned}\end{equation}
As for the lower bound, we need to have $s\gg \max\{\sigma^{-4},\sigma^4\}$ to apply the approximations.

Very similarly to \eqref{1-FGUEintegral}, one gets that the exponent is maximal for $\mu=\mu_0=\frac{3\sigma^4}{c+3\sigma^4}$. We also have
\begin{equation}
-\frac43\sqrt{\frac43\frac c{4\sigma^4}}\mu^{3/2}-\frac43(1-\mu)^{3/2}
=-\frac43\frac{\sqrt{c}}{\sqrt{c+3\sigma^4}}-\frac{(c+3\sigma^4)^{3/2}}{6\sigma^4\sqrt{c}} (\mu-\mu_0)^2+\O\left((\mu-\mu_0)^3\right).
\end{equation}
This gives
\begin{equation}\label{eq4.16}
1-F^{(\sigma)}(s)\le C'\frac{\sigma^4}{\sqrt{c+3\sigma^4}\sqrt{c}}
\frac{\ln(s/(1-c))}{\sqrt{1-c}}e^{-\frac43\frac{\sqrt{c}}{\sqrt{c+3\sigma^4}}s^{3/2}}(1+\O(s^{-1/4}))
\end{equation}
for some constant $C'$ which does not depend on $c$ and $\sigma$. Finally, since
\begin{equation}
\frac{\sqrt{c}}{\sqrt{c+3\sigma^4}}= \frac{1}{\sqrt{1+3\sigma^4}}-\frac{3\sigma^4}{2(1+3\sigma^4)^{3/2}}(1-c)+\O((1-c)^2),
\end{equation}
we choose $1-c=\tilde c s^{-3/2}$. With the choice $\tilde c=\frac14 (1+3\sigma^4)^{3/2}/\sigma^4$, together with \eqref{eq4.16} we obtain
\begin{equation}
1-F^{(\sigma)}(s)\le C'' \sigma^6(1+3\sigma^4)^{-2} s^{3/4}\ln(s)\,e^{-\frac43\frac{1}{\sqrt{1+3\sigma^4}}s^{3/2}}
\end{equation}
for some other constant $C''$ independent of $\sigma,s$.
\end{proof}

\section{Tail bounds for deterministic initial profile}\label{s:deterministic}
In this section we prove Theorem~\ref{thm:generalcurve} confirming the heuristics that the leading contribution for the right tail decay comes from the position where the function $h_0(t)-t^2$ is maximal.

\begin{proof}[Proof of Theorem~\ref{thm:generalcurve}]
Let $\tau\in\R$ be a time such that $\kappa(h_0)=\sup_{t\in\R} \{h_0(t)-t^2\}=h_0(\tau)-\tau^2$.
For the lower bound in \eqref{generalcurvebound} note that
\begin{equation}
\sup_{t\in\R}\left\{h_0(t)+\A_2(t)-t^2\right\}\ge h_0(\tau)+\A_2(\tau)-\tau^2=\kappa(h_0)+\A_2(\tau)
\end{equation}
by the definition of the time $\tau$. Hence
\begin{equation}
\P\left(\sup_{t\in\R}\left\{h_0(t)+\A_2(t)-t^2\right\}\geq s\right)
\geq\P\left(\kappa(h_0)+\A_2(\tau)\geq s\right)=1-F_{\rm GUE}(s-\kappa(h_0)).
\end{equation}
This inequality together with the asymptotic \eqref{eqGUEtail} leads to the lower bound in \eqref{generalcurvebound}.

Now we consider the upper bound.
The function $h_0(t)-t^2$ is bounded from above by $\kappa(h_0)$ for all times $t\in\R$ and it is bounded from above by $\kappa(h_0)-\frac\varepsilon2 t^2$ for $|t|>M$.
Therefore
\begin{equation}\label{eq2.4}\begin{aligned}
&\P\left(\sup_{t\in\R}\left\{h_0(t)+\A_2(t)-t^2\right\}\geq s\right)\\
&\leq \P\bigg(\sup_{|t|\leq M}\left\{\kappa(h_0)+\A_2(t)\right\}\geq s\bigg)+\P\bigg(\sup_{|t|>M}\left\{\kappa(h_0)+\A_2(t)-\frac\varepsilon2 t^2\right\}\geq s\bigg)\\
&\leq \P\bigg(\sup_{|t|\leq M}\A_2(t)\geq s-\kappa(h_0)\bigg)+\P\bigg(\sup_{t\in\R}\left\{\A_2(t)-\frac\varepsilon2 t^2\right\}\geq s-\kappa(h_0)\bigg).
\end{aligned}\end{equation}
The first term is bounded using Lemma~\ref{lemma:Airyfinite}.
The second term is bounded using Theorem~\ref{thm:Airyparabola}.
Altogether we get
\begin{multline}
\P\left(\sup_{t\in\R}\left\{h_0(t)+\A_2(t)-t^2\right\}\geq s\right)\\
\le C'\frac{2M}{(s-\kappa(h_0))^{1/4}}e^{-\frac43 (s-\kappa(h_0))^{3/2}}
+C\frac{\ln[2(s-\kappa(h_0))/\varepsilon]}{(s-\kappa(h_0))^{3/4}\sqrt{\varepsilon/2}}e^{-\frac43(s-\kappa(h_0))^{3/2}}.
\end{multline}
Since $\varepsilon$ is fixed, for large $s$ the second term is smaller than the first one, which completes the proof.
\end{proof}


\begin{thebibliography}{10}

\bibitem{AS84}
M.~Abramowitz and I.A. Stegun.
\newblock {\em Pocketbook of Mathematical Functions}.
\newblock Verlag Harri Deutsch, Thun-Frankfurt am Main, 1984.

\bibitem{BBdF08}
J.~Baik, R.~Buckingham, and J.~Di Franco.
\newblock {Asymptotics of Tracy--Widom distributions and the total integral of
  a Painleve II function}.
\newblock {\em Comm. Math. Phys.}, 280:463--497, 2008.

\bibitem{BL13}
J.~Baik and Z.~Liu.
\newblock {On the average of the Airy process and its time reversal}.
\newblock {\em Electron. Commun. Probab.}, 18:1--10, 2013.

\bibitem{BR00}
J.~Baik and E.M. Rains.
\newblock Limiting distributions for a polynuclear growth model with external
  sources.
\newblock {\em J. Stat. Phys.}, 100:523--542, 2000.

\bibitem{CHH19}
J.~Calvert, A.~Hammond, and M.~Hedge.
\newblock {Brownian structure in the KPZ fixed point}.
\newblock {\em arXiv:1912.00992}, 2019.

\bibitem{CFS16}
S.~Chhita, P.L. Ferrari, and H.~Spohn.
\newblock {Limit distributions for KPZ growth models with spatially homogeneous
  random initial conditions}.
\newblock {\em Ann. Appl. Probab}, 28:1573--1603, 2018.

\bibitem{CGH19}
I.~Corwin, P.~Ghosal, and A.~Hammond.
\newblock {KPZ equation correlation in time}.
\newblock {\em arXiv:1907.09317}, 2019.

\bibitem{CLW16}
I.~Corwin, Z.~Liu, and D.~Wang.
\newblock {Fluctuations of TASEP and LPP with general initial data}.
\newblock {\em Ann. Appl. Probab.}, 26:2030--2082, 2016.

\bibitem{CQR11}
I.~Corwin, J.~Quastel, and D.~Remenik.
\newblock {Continuum statistics of the Airy$_2$ process}.
\newblock {\em Comm. Math. Phys.}, 317:347--362, 2013.

\bibitem{DOV18}
D.~Dauvergne, J.~Ortmann, and B.~Vir{\'a}g.
\newblock The directed landscape.
\newblock {\em arXiv:1812.00309}, 2018.

\bibitem{FO17}
P.L. Ferrari and A.~Occelli.
\newblock {Universality of the GOE Tracy--Widom distribution for TASEP with
  arbitrary particle density}.
\newblock {\em Eletron. J. Probab.}, 23(51):1--24, 2018.

\bibitem{FS05b}
P.L. Ferrari and H.~Spohn.
\newblock {A determinantal formula for the GOE Tracy-Widom distribution}.
\newblock {\em J. Phys. A}, 38:L557--L561, 2005.

\bibitem{Gro89}
P.~Groeneboom.
\newblock Brownian motion with a parabolic drift and {Airy} functions.
\newblock {\em Probab. Theory Related Fields}, 81:79--109, 1989.

\bibitem{Gro10}
P.~Groeneboom.
\newblock {The maximum of Brownian motion minus a parabola}.
\newblock {\em Electron. J. Probab.}, 15:1930--1937, 2010.

\bibitem{JLML10}
S.~Janson, G.~Louchard, and A.~Martin-Löf.
\newblock The maximum of {B}rownian motion with parabolic drift.
\newblock {\em Electron. J. Probab.}, 15:1893--1929, 2010.

\bibitem{Jo03b}
K.~Johansson.
\newblock Discrete polynuclear growth and determinantal processes.
\newblock {\em Comm. Math. Phys.}, 242:277--329, 2003.

\bibitem{MQR17}
K.~Matetski, J.~Quastel, and D.~Remenik.
\newblock {The KPZ fixed point}.
\newblock {\em preprint: arXiv:1701.00018}, 2017.

\bibitem{MS17}
B.~Meerson and J.~Schmidt.
\newblock Height distribution tails in the
  {K}ardar{\textendash}{P}arisi{\textendash}{Z}hang equation with {B}rownian
  initial conditions.
\newblock {\em J. Stat. Mech.}, 2017(10):103207, 2017.

\bibitem{NQR20}
M.~Nica, J.~Quastel, and D.~Remenik.
\newblock {One-sided reflected Brownian motions and the KPZ fixed point}.
\newblock {\em arXiv:2002.02922}, 2020.

\bibitem{Pim17}
L.P.R. Pimentel.
\newblock {Ergodicity of the KPZ fixed point}.
\newblock {\em arXiv:1708.06006}, 2017.

\bibitem{Pim19}
L.P.R. Pimentel.
\newblock {Brownian aspects of the KPZ fixed point}.
\newblock {\em arXiv:1912.11712}, 2019.

\bibitem{PS00}
M.~Pr{\"a}hofer and H.~Spohn.
\newblock Universal distributions for growth processes in 1+1 dimensions and
  random matrices.
\newblock {\em Phys. Rev. Lett.}, 84:4882--4885, 2000.

\bibitem{PS02}
M.~Pr{\"a}hofer and H.~Spohn.
\newblock Scale invariance of the {PNG} droplet and the {A}iry process.
\newblock {\em J. Stat. Phys.}, 108:1071--1106, 2002.

\bibitem{QR13B}
J.~Quastel and D.~Remenik.
\newblock Supremum of the {A}iry$_2$ process minus a parabola on a half line.
\newblock {\em J. Stat. Phys.}, 150:442--456, 2013.

\bibitem{QR13}
J.~Quastel and D.~Remenik.
\newblock {Airy processes and variational problems}.
\newblock In A.~Ram\'irez, G.~Ben Arous, P.~Ferrari, C.~Newman,
  V.~Sidoravicius, and M.~Vares, editors, {\em Topics in Percolative and
  Disordered Systems}. Springer, 2014.

\bibitem{QR16}
J.~Quastel and D.~Remenik.
\newblock {How flat is flat in a random interface growth?}
\newblock {\em Trans. Amer. Math. Soc.}, 371:6047--6085, 2019.

\bibitem{TW96}
C.A. Tracy and H.~Widom.
\newblock On orthogonal and symplectic matrix ensembles.
\newblock {\em Comm. Math. Phys.}, 177:727--754, 1996.

\bibitem{Vai20}
M.~Vaisband.
\newblock {Fluctuations of KPZ interfaces}.
\newblock Master thesis, Bonn University, 2020.

\end{thebibliography}

\end{document}